\documentclass[12pt]{amsart}
\usepackage{amsmath,amssymb,mathtools,mathabx,rotating}
\usepackage{graphicx}
\usepackage{color}
\usepackage{epstopdf}
\newcommand*{\id}{{\mathrm{id}}}

\newcommand*{\Kb}{{\mathbb K}}

\newcommand*{\CC}{{\mathbb{C}}}

\newcommand*{\un}{{\mathbf 1}}

\def\shuff#1#2{\mathbin{
      \hbox{\vbox{\hbox{\vrule \hskip#2 \vrule height#1 width 0pt}\hrule}\vbox{\hbox{\vrule \hskip#2 \vrule height#1 width 0pt\vrule }\hrule}}}}
\def\shuffl{{\mathchoice{\shuff{5pt}{3.5pt}}{\shuff{5pt}{3.5pt}}{\shuff{3pt}{2.6pt}}{\shuff{3pt}{2.6pt}}}}
\def\shuffle{{\, \shuffl \,}}

\newtheorem{thm}{Theorem}
\newtheorem{exam}{Example}

\newtheorem{cor}[thm]{Corollary}

\newtheorem{lem}[thm]{Lemma}
\newtheorem{prop}[thm]{Proposition}
\newtheorem{defn}{Definition}
\newtheorem{rmk}[thm]{Remark}

\begin{document}


\title[c-free cumulants and shuffle algebra]{A group-theoretical approach\\ to conditionally free cumulants}


\author[K.~Ebrahimi-Fard]{Kurusch Ebrahimi-Fard}
\address{Department of Mathematical Sciences, 
		Norwegian University of Science and Technology (NTNU),
		NO-7491 Trondheim, Norway.}
         \email{kurusch.ebrahimi-fard@ntnu.no}         
         \urladdr{https://folk.ntnu.no/kurusche/}

\author[F.~Patras]{Fr\'ed\'eric Patras}
\address{Univ.~de Nice,
Labo.~J.-A.~Dieudonn\'e,
         		UMR 7351, CNRS,
         		Parc Valrose,
         		06108 Nice Cedex 02, France.}
\email{patras@math.unice.fr}
\urladdr{www-math.unice.fr/$\sim$patras}


\begin{abstract}
In this work we extend the recently introduced group-theoretical approach to moment-cumulant relations in non-commutative probability theory to the notion of conditionally free cumulants. This approach is based on a particular combinatorial Hopf algebra which may be characterised as a non-cocommutative generalisation of the classical unshuffle Hopf algebra. Central to our work is the resulting non-commutative shuffle algebra structure on the graded dual. It implies an extension of the classical relation between the group of Hopf algebra characters and its Lie algebra of infinitesimal characters and, among others, the appearance of new forms of ``adjoint actions'' of the group on its Lie algebra which happens to play a key role in the new algebraic understanding of conditionally free cumulants.
\end{abstract}


\maketitle





\tableofcontents


\section{Introduction}
\label{sec:intro}

Voiculescu \cite{Voiculescu92,Voiculescu95} introduced in the 1980s the theory of free probability. It is based on the notion of free independence, or freeness, that is, the absence of algebraic relations. Free cumulants encode the notion of free independence. Speicher \cite{Speicher97c} uncovered a combinatorial approach to free cumulants based on the lattice of non-crossing set partitions and its M\"obius calculus. The reader is referred to \cite{MingoSpeicher17,NicaSpeicher06,Speicher17} for introductions and reviews. The free moment-cumulant relation for the $n$-th univariate moment $m_n$ is given by 
\begin{equation}
\label{freecumulants}
	m_n = \sum_{\pi \in NC_n} \prod_{\pi_i \in \pi} k_{|\pi_i|},
\end{equation}
where $NC_n$ is the lattice of non-crossing set partitions $\pi:=\{\pi_1,\ldots,\pi_j\}$ of  $[n]:=\{1,\dots,n\}$ and $|\pi_i|$ denotes the number of elements in the block $\pi_i \in \pi$. Here $k_l$ denotes the $l$-th free cumulant. Analogous statements hold for monotone \cite{HasebeSaigo11} and boolean cumulants \cite{Speicher97b}. Indeed, for monotone cumulants $h_l$ one has the monotone moment-cumulant relation
\begin{equation}
\label{monotonecumulants}
	m_n = \sum_{\pi \in NC_n}\frac{1}{\tau(\pi)!} \prod_{\pi_i \in \pi} h_{|\pi_i|}.
\end{equation}
The so-called tree (forest) factorial $\tau(\pi)!$ corresponds to the forest $\tau(\pi)$ of rooted trees associated to the nesting of the blocks of the non-crossing partition $\pi \in NC_n$. See \cite{Arizmendi15} for details. Boolean cumulants $r_l$ satisfy the boolean moment-cumulant relation
\begin{equation}
\label{booleancumulants}
	m_n=\sum_{I \in B_n} \prod_{l_i \in I} r_{|l_i|},
\end{equation}
where $B_n$ is the boolean lattice of interval partitions. The multivariate generalisations of these moment-cumulant relations will be given further below in the shuffle algebra setting. Relations between the different cumulants have been studied in great detail by Arizmendi et al.~in the recent article \cite{Arizmendi15}. 

\smallskip

The framework for our group-theoretical approach to free, boolean and monotone moment-cumulant relations has been developed in a series of recent works \cite{EFP15,EFP16,EFP16a,EFP17,EFP18}. It is based on a particular graded, connected, non-commutative, non-cocommutative word Hopf algebra $H$ defined on the double tensor algebra over a non-commutative probability space $(A,\varphi)$ with linear unital map $\varphi \colon A \to \mathbb{K}$. One may characterise $H$ as a non-cocommutative generalisation of the classical cocommutative unshuffle Hopf algebra \cite{Reutenauer93}. 

In \cite{EFP15} we first defined the coproduct of $H$. It is right-sided (right-handed in Turaev's original terminology) \cite{LodRonc,Menous18,Turaev} and splits into left and right half-coproducts, which define the structure of (un)shuffle bialgebra on $H$ \cite{Foissy07}. This implies, on the other hand, a splitting of the associative convolution product on the graded dual $H^\ast$ into two non-associative half-shuffles, which define a non-commutative shuffle algebra (aka dendriform algebra) structure on $H^*$. 
It follows from the existence of these structures that, besides the classical exponential $\exp^*$, two other exponential-type maps, denoted $\mathcal{E}_\prec$ and $\mathcal{E}_\succ$, can be defined in terms of the two half-shuffles. This allows for a refinement of the classical correspondence between a group and its Lie algebra, by means of the exponential map since \it all \rm three maps define bijections from the Lie algebra $g \subset H^*$ of infinitesimal characters to the group $G \subset H^*$ of algebra characters on $H$. Once a particular character $\Phi \in G$ has been defined in terms of the natural extension of $\varphi$ from $A$ to $H$, monotone, free, and boolean cumulants can be considered as infinitesimal characters $\rho$, $\kappa$, $\beta$ in $g$, respectively, and are defined in terms of the identities
\begin{equation}
\label{key1}
	\Phi = \exp^*\!(\rho) = \mathcal{E}_\prec(\kappa) = \mathcal{E}_\succ(\beta).
\end{equation}
In \cite{EFP18} we showed that, indeed, the left and right half-shuffle exponentials, $\mathcal{E}_\prec(\kappa)$ respectively $\mathcal{E}_\succ(\beta)$, give rise to free respectively boolean multivariate moment-cumulant relations \cite{Speicher97b,Speicher97c}. The shuffle exponential $\exp^*\!(\rho)$ describes instead the monotone moment-cumulant relations \cite{HasebeSaigo11}. This yields a novel and unifying approach to moment-cumulant relations in non-commutative probability. In \cite{EFP18} we also showed how three logarithm-type maps ($R$-transformations) corresponding to these exponentials together with a particular shuffle group-theoretical  adjoint operation permit to recover relations between cumulants, which were described explicitly in \cite{Arizmendi15} using classical M\"obius calculus on non-crossing set partitions. Regarding the latter an important remark is in order. In our group-theoretical approach non-crossing partitions enter the picture only through the closed formulas for the evaluation of the exponential-type maps $\exp^*$, $\mathcal{E}_\prec$ and $\mathcal{E}_\succ$ on words from $H$. For instance, the righthand side of the univariate free moment-cumulant relation \eqref{freecumulants} results from calculating $\mathcal{E}_\prec(\kappa)(w)$ for a word in $w=a^{\otimes n} \in H$ of length $n$ in a single letter, i.e., random variable $a \in A$.        

There is a well-known connection between moment-cumulant relations in classical probability and relations between Green's respectively connected Green's functions (e.g., in perturbative quantum field theory (QFT)). It extends to free moment-cumulant relations and Green's respectively connected Green's functions relations in planar QFT (see, e.g., \cite{EFP16a}). Still in perturbative QFT, replacing the vacuum by other, non-trivial, ground states gives rise to interesting phenomena meaningful to applications, see, e.g.~\cite{frab,bp,gott}. The notion of conditionally free probability, introduced in \cite{Bozejko96} by Bo\.{z}ejko et al., shares with pQFT over non-trivial vacua a key feature (the analogy stops there at the moment but deserves to be further analysed): the idea to consider a theory of free probabilities \it relative \rm to a given arbitrary state or, equivalently, to consider the behaviour of a pair of states (where, however, the two states have different roles). This conditional extension of Voiculescu's theory allows, among others, for the definition of a conditionally (or c-)free convolution product and $R$-transform, for the explicit calculation of distributions of conditionally free Gaussian and free Poisson distributions and other similar key behavioural properties one expects for a generalised free probability theory. 

The paper on hand shows how the combinatorial side of c-free cumulants is naturally captured by the group-theoretical picture sketched above. In this respect the aforementioned shuffle group-theoretical adjoint action, which permits to express monotone, free, and boolean cumulants in terms of each other, is the central object. Indeed, we will show how it allows to relate c-free cumulants with free and boolean, and therefore also with monotone cumulants.  

Finally, we remark that several works have appeared in recent years, applying Hopf algebra techniques in the context of free probability \cite{FriedrichMckay13, FriedrichMckay15,ManzelSchuermann17,MastnakNica10}. Moreover, non-commutative shuffle algebras appeared in the work by Belinschi et al.~\cite{BBLS11} in relation to the problem of the infinite divisibility of the normal distribution with respect to additive convolution in free probability. However, our approach is rather different, and potential connections have to be explored in the future.

\smallskip

The paper is organised as follows. In Section \ref{sect:condit} we survey briefly the foundations of the theory of c-free probabilities. In Section \ref{sec:shufflealgebra} we present the necessary background on non-commutative shuffle algebras together with the particular combinatorial Hopf algebra (denoted $H$ in the Introduction) as main example. This Hopf algebra will provide the underlying framework for our shuffle group-theoretical approach to moment-cumulant relations. Section \ref{sec:explog} introduces the three exponential bijections, which provide the group-theoretical setting for free, boolean and monotone moment-cumulant relations. The next section recalls the shuffle algebra approach to  the latter. Section \ref{sec:condfree} contain the main result of the paper. It describes conditionally free cumulants and convolution using the group-theoretical machine introduced in Sections \ref{sec:explog} and \ref{sec:cumulants}.

\medskip

\noindent\textbf{Acknowledgements:} We would like to thank the organisers of the CARMA 2017 workshop at CIRM in Luminy and the CNRS PICS project: Alg\`ebres de Hopf combinatoires et probabilit\'es non commutatives for its support.


\section{Conditional freeness}
\label{sect:condit}

We first fix some notation. Let $NC_n$ denote the lattice of non-crossing set partitions of order $n$ and $B_n$ is the boolean lattice of interval partitions of order $n$. See \cite{NicaSpeicher06} for details. Recall that a partition $\pi=\pi_1 \sqcup \dots \sqcup \pi_k$ of $[n]$ is non-crossing if and only if there are no $(i,l)\in \pi_a \times \pi_a$ and $(j,m)\in \pi_b \times \pi_b$, $1\leq a,b\leq k$ with $a\not= b$ and $i<j<l<m$. A partition is boolean if each of its blocks $\pi_i$ is an interval, i.e., if $a,b\in \pi_i$ and $a<c<b$, then $c\in\pi_i$. We let $|\pi_i|$ denote the number of elements in the block $\pi_i \in \pi$. The elements in each block $\pi_i \in \pi$ are naturally ordered, $\pi_i:=\{j_1 < \cdots < j_{|\pi_i|}\}$, and we set $a_{\pi_i}:=a_{j_1} \cdots a_{j_{|\pi_i|}}$. 
A
Conditionally free probability generalises the fundamental notion of Voiculescu's free probability theory in the context of two states. We introduce here briefly the main constructions relevant for our later purposes and refer the reader to the original article \cite{Bozejko96} for details. 

We work in the framework of unital algebras $A$ and a state on $A$ simply means a unital linear form $\varphi:{A} \to \CC$ (unital meaning that $\varphi(1)=1$). Given now unital algebras ${A}_i,\ i\in I$, each equipped with a pair of states $(\varphi_i,\psi_i)$, their free product reads
$$
	\ast_{i\in I}{A}_i\cong \CC\cdot 1\oplus\bigoplus\limits_{n=1}^\infty \bigoplus\limits_{i(1)\not
	= \dots\not= i(n)}{A}_{i(1)}^+\otimes\dots\otimes {A}_{i(n)}^+,
$$
where ${A}_j^+:= \ker \psi_j$. A new state $\varphi$ is then defined on $\ast_{i\in I}{A}_i$ by requiring $\varphi(1)=1$ and 
$$
	\varphi(a_1\cdots a_n)=\varphi_{i(1)}(a_1)\cdots \varphi_{i(n)}(a_n),
$$
when $a_j\in {A}_{i(j)}^+$ and $i(1)\not=\dots \not= i(n).$ Notice that the role of the $\varphi_i$ and that of the $\psi_i$ are not symmetrical.

Suppose that ${A}=\CC \langle X \rangle$. The (free) moments of the states $\varphi$ and $\psi$ on ${A}$ are defined by
$$
	m_n^\varphi:=\varphi(X^n),\quad m_n^\psi:=\psi(X^n).
$$
Speicher's free moment-cumulant relations (which define implicitly the free cumulants $k_n^\psi$) are then given by
$$
	m_n^\psi=\sum\limits_{p=1}^n\mathop{\sum\limits_{l(1),\dots,l(p)\geq 0  }}_{l(1)+\dots +l(p)=n-p}
	k_p^\psi m_{l(1)}^\psi\cdots m_{l(p)}^\psi
$$ 
or, in terms of non-crossing partitions:
$$
	m_n^\psi = \sum\limits_{\pi\in NC_n}\prod\limits_{\pi_l\in \pi}k_{|\pi_l|}^\psi.
$$
Conditionally (c-)free cumulants $R_k^{(\varphi,\psi)}$ for the pair $(\varphi,\psi)$ are defined instead by:
$$
	m_n^\varphi=\sum\limits_{j=1}^n\mathop{\sum\limits_{l(1),\ldots,l(j)\geq 0  }}_{l(1)+\cdots +l(j)=n-j}
	R_j^{(\varphi,\psi)} m_{l(1)}^\psi \cdots m_{l(j-1)}^\psi m_{l(j)}^\varphi,
$$
or, in terms of non-crossing partitions:
\begin{equation}\label{cfreecum}
	m_n^\varphi =\sum_{\pi \in NC_n}\prod_{\pi_l \in \pi \atop \pi_l\, \mathrm{inner}}
	k_{|\pi_l|}^\psi\prod_{\pi_l \in \pi \atop \pi_l\, \mathrm{outer}}R_{|\pi_l|}^{(\varphi,\psi)}.
\end{equation}
Here, the terms ``outer'' and ``inner'' refer to the structure of non-crossing partitions. A block $\pi_i$ of $\pi \in NC_n$ is inner if there exists a $\pi_j$ and $a,b\in \pi_j$ such that $a<c<b$ for all $c\in \pi_i$. A block which is not inner is outer.

As in classical free probability, cumulants characterise free convolution in the sense that the distribution of the c-free convolution 
$$
	(\varphi,\psi)=(\varphi_1,\psi_1) \boxplus (\varphi_2,\psi_2)
$$
of two pairs of states on $\CC  \langle X_1 \rangle$ respectively $\CC  \langle X_2 \rangle$ is characterized by
\begin{align}
	k_n^\psi 		&= k_n^{\psi_1} + k_n^{\psi_2}, 			\label{conv1}\\
	R_n^{(\varphi,\psi)}  &= R_n^{(\varphi_1,\psi_1)} + R_n^{(\varphi_2,\psi_2)}.		\label{conv2}
\end{align}


\section{Shuffle algebra}
\label{sec:shufflealgebra}

Card shufflings appeared already in Poincar\'e's treatise on probability. Later, so-called perfect shuffles led to the definition of commutative shuffle products. They were axiomatised independently by Eilenberg and MacLane in 1953 \cite{EM53} and  Sch\"utzenberger in 1958 \cite{Schutzenberger58}, in relation to the homology of commutative algebras respectively combinatorics and classical Lie algebra theory. Eilenberg and MacLane studied also non-commutative shuffle products using the idea of splitting such products into left and right half-shuffle products. The resulting algebraic shuffle relations for those half-shuffles allowed them to demonstrate abstractly the associativity of shuffle products familiar in topology. The commutative shuffle product is essential in many fields of pure and applied mathematics. In Chen's work for example \cite{Chen57,Chen71}, it encodes algebraically the product of iterated integrals. The splitting into half-shuffles reflects the integration by parts rule for Riemann integrals of ordinary functions. Reutenauer's classic monograph \cite{Reutenauer93} embedded Chen's fundamental work into  the setting of connected graded cocommutative Hopf algebras. Together with its origins in topology, combinatorics and Lie theory, these phenomena explain the ubiquity of shuffles. From 2001 onwards non-commutative shuffle products were axiomatised and explored systematically by Loday, Ronco, Chapoton, and others \cite{Loday01}, who called them dendriform\footnote{We prefer, for historical and conceptual reasons, the classical name shuffle.} products. These works paved the way to many theoretical developments and the discovery of new structures. Indeed, anti-symmetrising non-commutative shuffle products yields Lie brackets. However, for the half-shuffle products the picture is more subtle. Indeed, they combine to Lie admissible pre-Lie products \cite{Burde06,Cartier11,ChapLiv01,Manchon11}, which have been discovered independently in geometry, algebraic deformation theory and control theory~\cite{AG78,AG81,Gerstenhaber63,Vinberg63}. Regarding the following definitions and statements we refer the reader, for instance, to Manchon's survey \cite{Manchon11}.

\smallskip

\begin{defn}
A shuffle algebra $(D,\prec,\succ)$ consists of a $\mathbb{K}$-vector space $D$ together with products $\prec$ and $\succ$ called respectively the left and right half-shuffle products, satisfying the shuffle relations
\allowdisplaybreaks
\begin{eqnarray}
	(a \prec b) \prec c   &=& a \prec (b \succ c + b\prec c)        	\label{A1}\\
  	(a \succ b) \prec c   &=& a \succ (b \prec c)   				\label{A2}\\
   	a \succ (b \succ c)  &=& (a \succ b + a \prec b) \succ c        	\label{A3},
\end{eqnarray}
for $a,b,c \in D$. A commutative shuffle algebra is defined by including the extra relation
\begin{equation}
\label{comshuf}
	a \succ b - b \prec a = 0.
\end{equation}
\end{defn}
We will use from now on the terminology ``non-commutative shuffle algebra'' to emphasize explicitely non-commutativity, i.e., the absence in a given shuffle algebra of relation \eqref{comshuf}.
\begin{prop}
Let $(D,\prec,\succ)$ be a shuffle algebra. The shuffle product $m_{*}: D \otimes D \to D$, $m_{*}(a \otimes b)=:a * b$, defined in terms of the two half-shuffles
\begin{equation}
\label{dendassoc}
	a * b:= a \succ b + a \prec b,     
\end{equation}
for $a,b \in D$, is associative. In a commutative shuffle algebra the product $m_*$ is commutative. 
\end{prop}

\begin{defn}
A left pre-Lie algebra $(P,\rhd)$ consists of a $\mathbb{K}$-vector space $P$ with a binary product $\rhd \colon P \otimes P \to P$ satisfying the left pre-Lie identity 
\begin{equation}
\label{leftpreLie}
	(a \rhd b) \rhd c - a \rhd (b \rhd c) = (b \rhd a) \rhd c - b \rhd (a \rhd c),     
\end{equation}
for $a,b,c \in P$. An analogous notion of right pre-Lie algebra exists.
\end{defn}

\begin{prop}
Let $(P,\rhd)$ be a left pre-Lie algebra. For $a,b \in P$ the commutator bracket $[a,b]:= a \rhd b - b \rhd a$ defines a Lie algebra on $P$. 
\end{prop}

\begin{prop}
Let $(D,\prec,\succ)$ be a shuffle algebra. For $a,b \in D$ the product 
\begin{equation}
\label{denpreLie} 
	a \rhd b:= a \succ b - b \prec a  
\end{equation}
defines a left pre-Lie algebra on $D$. 
\end{prop}

\noindent Moreover, one verifies quickly that $[a,b] =a \rhd b - b \rhd a= a * b - b * a$ in $(D,\prec,\succ)$. Observe that in a commutative shuffle algebra the pre-Lie product \eqref{denpreLie} becomes trivial. We define the left and right multiplication maps, $L_{a \prec}(b):= a \prec b$, $L_{a \succ}(b):= a \succ b$ and $R_{\prec a}(b):= b \prec a$, $R_{\succ a}(b):= b\succ a$. Combining them, we can write $L_{a\rhd}(b):=a \rhd b=(L_{a \succ} - R_{\prec a})(b)$. The maps $L_{a \succ}$ and $R_{\prec b}$ commute thanks to relation \eqref{A2}. 

A consistent definition of unital shuffle algebra demands some caution -- due to the fact that it is a priori difficult to split the identity $\un * \un=\un$ via half-shuffles. This said, the augmentation of the shuffle algebra $(D,\prec,\succ)$ by a unit $\un$ to $\overline D := D \oplus \mathbb{K}.\un$ is defined by requiring for any $a \in D$, that $a * \un := \un * a := a$. Moreover, concerning the half-shuffles we define 
$$
	\un \succ a := a =: a \prec \un,
$$ 
and $\un \prec a := 0 =: a \succ \un$. This is further extended to include the pre-Lie product, i.e.,  $a \rhd \un := a =: \un \rhd a$. However, it is important to note that the separate cases of  $\un \prec \un$ and $\un \succ \un$ must be excluded as they can not be defined consistently. Two examples of (unital) shuffle algebras are given next. 

\begin{exam}\label{ex:shuffleAlg}{\rm{
The non-unital tensor algebra over a $\mathbb{K}$-vector space $A$ is defined by
$$
	T_+(A):=\bigoplus_{n > 0}A^{\otimes n}.
$$ 
Elements in $T_+(A)$ are denoted by words $w=a_{i_1} \cdots a_{i_m} \in A^{\otimes m}$. The number $|w|$ of letters of a word $w \in T_+(A)$ defines its length. The unital tensor algebra $T(A):=\bigoplus_{n \ge 0}A^{\otimes n}$ is defined by adding the empty word $\un$ in $A^{\otimes 0} \subset T(A)$. The commutative and associative shuffle product on words is defined iteratively on $T(A)$ by $w \shuffle \un = w = \un \shuffle w$ and
\begin{equation}
\label{shuffleproduct}
	a_{i_1} \cdots a_{i_m} \shuffle a_{j_1} \cdots a_{j_l} :=
	a_{i_1} (a_{i_2} \cdots a_{i_m}\shuffle a_{j_1} \cdots a_{j_l}) 
	+ a_{j_1} (a_{i_1} \cdots a_{i_m}\shuffle a_{j_2} \cdots a_{j_l}), 
\end{equation}
for any words $a_{i_1} \cdots a_{i_m},a_{j_1} \cdots a_{j_l} \in T_+(A)$. The two terms on the righthand side of  \eqref{shuffleproduct} define respectively the left and right half-shuffles satisfying \eqref{A1}-\eqref{A3} and \eqref{comshuf}. Note that there exists a natural  grading on $T(A)$ given by the length of words.}}
\end{exam}

Next we present an example of a non-commutative shuffle algebra. It provides the framework for our  approach to non-commutative probability and consists of the double tensor algebra over a $\mathbb{K}$-vector space $A$. Further below, the latter is supposed to be a unital $\mathbb{K}$-algebra, which together with the linear unital map $\varphi \colon A \to \mathbb{K}$ defines a non-commutative probability space.

\begin{exam}\cite{EFP15}\label{ex:HopfAlg}{\rm{
The non-unital double tensor algebra over a $\mathbb{K}$-vector space $A$ is defined   by
$$
	T_+(T_+(A)):=\bigoplus_{n > 0} T_+(A)^{\otimes n}.
$$ 
We use the bar-notation to denote elements $w_1 | \cdots | w_n \in T_+(T_+(A))$, where $w_i \in T_+(A)$, $i=1,\ldots,n$. The space $T_+(T_+(A))$ is equipped with the concatenation product, defined for $w= w_1 | \cdots | w_n$ and $w'=  w_1' | \cdots | w_m'$ in $T_+(T_+(A))$ by $w|w' := w_1 | \cdots | w_n | w_1' | \cdots | w_m'$. This non-commutative algebra is multigraded, that is, $T_+(T_+(A))_{n_1,\ldots ,n_k}:=T_{n_1}(A)\otimes \cdots \otimes T_{n_k}(A)$, as well as graded. The degree $n$ part is  $T_+(T_+(A))_n:=\bigoplus_{n_1+ \cdots +n_k=n}T(T(A))_{n_1,\ldots ,n_k}$. Similar observations hold for the unital case, that is, $T(T(A))=\oplus_{n \ge 0} T(A)^{\otimes n}$, and we will identify without further comments a bar symbol such as $w_1|1|w_2$ with $w_1|w_2$. The empty word, which is the unit for the bar-product, is denoted $\un \ \in T(T_+(A))_0$.

\medskip

Given two (canonically ordered) subsets $S \subseteq U$ of the set of integers $\mathbb{N}$, we call connected component of $S$ relative to $U$ a maximal sequence $s_1, \ldots , s_n$ in $S$, such that there are no $ 1 \leq i < n$ and $u \in U$, such that $s_i < u <s_{i+1}$. In particular, a connected component of $S$ in $\mathbb{N}$ is simply a maximal sequence of successive elements $s,s+1,\ldots ,s+n$ in $S$. Consider a word $a_1 \cdots a_n \in T_+(A)$. For the  (canonically ordered) non-empty set $S:=\{s_1,\ldots, s_p\} \subseteq [n]$, we define 
\begin{equation}
\label{aS}
	a_S:= a_{s_1} \cdots a_{s_p},
\end{equation}
and $a_\emptyset:=\un$. Denoting by $J_1,\ldots,J_k $ the connected components of $[n] - S$, we then set 
\begin{equation}
\label{aJ}
	a_{J^S_{[n]}}:= a_{J_1} | \cdots | a_{J_k}. 
\end{equation}
More generally, for $S \subseteq U \subseteq [n]$, set  $a_{J^S_U}:= a_{J_1} | \cdots | a_{J_k}$, where the $a_{J_i}$ are now the connected components of $U-S$ in $U$. We remark that the bar-notation in \eqref{aJ} respectively in $a_{J^S_U}$ may be interpreted as marking the places where sequences of consecutive letters have been extracted from a word.  Using \eqref{aS} and \eqref{aJ} we define a coproduct on $T(T_+(A))$. 

\begin{defn}\label{def:coproduct}
The coproduct $\Delta : T(A) \to T(A) \otimes  T(T_+(A))$ is defined by $\Delta(\un):= \un \otimes \un$ and
\begin{equation}
\label{HopfAlg}
	\Delta(a_1\cdots a_n) :=\sum_{S \subseteq [n]} a_S \otimes a_{J^S_{[n]}}.
\end{equation} 
It is extended multiplicatively to all of $T(T_+(A))$, i.e., $\Delta(w_1 | \cdots | w_m) := \Delta(w_1) \cdots \Delta(w_m).$ 
\end{defn}

For example, the coproduct of a single letter is $\Delta(a) = a \otimes \un + \un \otimes a$. For a word $a_1a_2$ of length two it is 
$$
	\Delta(a_1a_2) = a_1a_2 \otimes \un + \un \otimes a_1a_2 + a_1 \otimes a_2 + a_2 \otimes a_1.
$$
For the word $a_1a_2a_3 \in A^{\otimes 3}$ we calculate
\allowdisplaybreaks
\begin{align*}
	\Delta(a_1a_2a_3) &= a_1a_2a_3 \otimes \un + \un \otimes a_1a_2a_3 
	 + a_1 \otimes a_2a_3 + a_2 \otimes a_1|a_3 + a_3 \otimes a_1a_2 \\
	&\qquad
	+ a_1a_2 \otimes a_3 
	+ a_1a_3 \otimes a_2 
	+ a_2a_3 \otimes a_1. 
\end{align*}
The coproduct $\Delta(a_1 \cdots a_6)$ includes among others the sum 
$$
	a_1a_3a_5 \otimes a_2 | a_4 | a_6 
	+ a_3a_4a_5 \otimes a_1 a_2 | a_6 
	+ a_3a_6 \otimes a_1a_2 | a_4 a_5 
	+ a_4a_5a_6 \otimes a_1a_2 a_3 . 
$$

The proofs of the following two theorems appeared in \cite{EFP15}. See also \cite{EFP16}.
 
\begin{thm} \label{thm:HA} \cite{EFP15}
The graded algebra $H:=T(T_+(A))$ equipped with the coproduct \eqref{HopfAlg} is a connected graded non-commutative and non-cocommutative Hopf algebra. 
\end{thm}

The central observation in \cite{EFP15} consist of the splitting of the coproduct \eqref{HopfAlg} into two parts
$$
	\Delta=\Delta_{\prec} + \Delta_{\succ}. 
$$
The corresponding left respectively right half-coproducts are defined on $H$ by
\begin{align}
	\Delta_{\prec}(a_1 \cdots a_n) 
				&:= \sum_{1 \in S \subseteq [n]} a_S \otimes a_{J^S_{[n]}}	
				=: \Delta^+_{\prec}(a_1 \cdots a_n) + a_1 \cdots a_n \otimes \un 	\label{lefthalfcoprod}
\end{align}
and
\begin{align}
	\Delta_{\succ}(a_1 \cdots a_n) 
			&:= \sum_{1 \notin S \subset [n]} a_S \otimes a_{J^S_{[n]}} 		
			=: \Delta^+_{\succ}(a_1 \cdots a_n) + \un \otimes a_1 \cdots a_n .	\label{righthalfcoprod}
\end{align}
Note that for $w \in H$ the reduced coproduct $\Delta^+(w):=\Delta(w) - w \otimes \un - \un \otimes w$ splits into
$$
	\Delta^+(w) = \Delta^+_\prec(w) + \Delta^+_\succ(w).
$$
For instance, the coproduct $\Delta(a_1a_2a_3)$ is the sum of the left half-coproduct 
$$
	\Delta_{\prec}(a_1a_2a_3) = a_1a_2a_3 \otimes \un 
	+ a_1 \otimes a_2a_3 + a_1a_2 \otimes a_3 + a_1a_3 \otimes a_2
$$
and right half-coproduct
$$
	\Delta_{\succ}(a_1a_2a_3) = \un \otimes a_1a_2a_3 
	+ a_2 \otimes a_1|a_3 + a_3 \otimes a_1a_2 + a_2a_3 \otimes a_1.
$$
The two half-coproducts are extended to $H$ by defining them on $w_1 | \cdots | w_m$
\begin{eqnarray*}
	\Delta_{\prec}(w_1 | \cdots | w_m) &:=& \Delta_{\prec}(w_1)\Delta(w_2) \cdots \Delta(w_m) \\
	\Delta_{\succ}(w_1 | \cdots | w_m) &:=& \Delta_{\succ}(w_1)\Delta(w_2) \cdots \Delta(w_m). 
\end{eqnarray*}

\begin{thm} \cite{EFP15} \label{thm:bialg} The algebra
$H$ equipped with $\Delta_{\prec}$ and $\Delta_{\succ}$ is a unital unshuffle bialgebra. 
\end{thm}

For details on the notion of unshuffle bialgebra we refer the reader to Foissy's article \cite{Foissy07} and, in the present context, to our previous articles, e.g., \cite{EFP15,EFP17}. Recall that the space of linear maps, $\mathrm{Lin}(H,\mathbb{K})$, is (as for all Hopf algebras) an associative and untial $\mathbb{K}$-algebra with respect to the non-commutative convolution product defined for $\Gamma,\Psi \in \mathrm{Lin}(H,\mathbb{K})$ in terms of the coproduct \eqref{HopfAlg} 
$$
	\Gamma*\Psi := m_\Kb (\Gamma \otimes \Psi) \Delta,
$$  
where $m_\Kb$ stands for the product map in $\mathbb{K}$. The augmentation map $e: H \rightarrow \mathbb{K}$, defined by $e(\un) :=1$ and zero on the so-called  augmentation ideal $H_+:=T_+(T_+(A))$, is the unit for this convolution product. In light of the splitting $\Delta^+ = \Delta^+_{\prec} + \Delta^+_{\succ}$ we define accordingly the left and right convolution half-products on $\mathrm{Lin}(H_+,\mathbb{K})$:
$$
	\Gamma \prec \Psi:=m_\mathbb{K}(\Gamma \otimes \Psi)\Delta^+_\prec \ 
	\qquad\
	\Gamma \succ \Psi:=m_\mathbb{K}(\Gamma \otimes \Psi)\Delta^+_\succ .
$$
These operations are extended by setting, for $\Psi\in \mathrm{Lin}(H_+,\mathbb{K})$, $e \prec \Psi:=0$, $e \succ \Psi:=\Psi$, $\Psi\prec e:=\Psi,\ \Psi\succ e:=0.$ As a result we obtain the next proposition.

\begin{prop} \cite{EFP15} \label{prop:shufflealgebra}
The space $\mathrm{Lin}(H,\mathbb{K})$ equipped with $(\prec, \succ)$ is a unital shuffle algebra.
\end{prop}
}}
\end{exam}


\section{Shuffle and half-shuffle exponentials and logarithms}
\label{sec:explog}
      
Recall the formal, i.e., purely algebraic component of the relations between a group and its Lie algebra as encoded in the Baker--Campbell--Hausdorff formula \cite{Reutenauer93}. The natural framework to understand these phenomena is provided by complete connected cocommutative Hopf algebras and, in particular, classical commutative shuffle Hopf algebras \cite{Reutenauer93}. In this case the exponential and logarithm maps relate the Lie algebra of primitive elements bijectively to the group of group-like elements. In \cite{EFP17} we started to explore how the classical correspondence between groups and Lie algebras, and related properties and identities, translate in the setting of the non-(co)commutative shuffle bialgebra $H$ in Theorem \ref{thm:HA}. It turns out that in this case one has to consider not only the usual exponential-logarithm correspondence but also two shuffle-type counterparts defined in terms of the two half-shuffle products. In this section we recall from \cite{EFP17} the shuffle and half-shuffle exponentials and logarithms and introduce the group-theoretical shuffle adjoint actions. 
      
\smallskip
      
A preliminary remark is in order regarding convergence issues. They are left aside in the present paper since we deal implicitly with formal series expansions over free shuffle algebras (insuring the convergence in the formal sense), or with graded algebras (in which case formal power series expansions restrict to finite expansions in each degree). In practice, ``let $D$ be a shuffle algebra'' means therefore till the end of the present section, ``let $D$ be a free or a graded connected (i.e.~with no degree zero component) shuffle algebra''.
      
\medskip 
   
Let $(\overline D,\succ,\prec)$ be a unital shuffle algebra. For any element $a \in D$ we define the usual exponential and logarithm in terms of the associative shuffle product \eqref{dendassoc}
\begin{equation}
\label{ExpLog}
	\exp^*\!(a):=\un + \sum_{n > 0} \frac{a^{* n}}{n!}  
	\qquad\ 
	\log^*(\un+a):=-\sum_{n>0}(-1)^n\frac{a^{*n}}{n}. 
\end{equation}
For $a \in D$ we define $a^{\succ{0}}:=\un=:a^{\prec{0}}$ and for $n>0$, $a^{\prec{n}}:=a \prec (a^{\prec{n-1}})$, and $a^{\succ{n}} := (a^{\succ{n-1}} )\succ a$. Then the left and right half-shuffle exponentials are defined for $a \in D$ 
$$
	\mathcal{E}_\prec(a) :=\un + \sum_{n > 0} a^{\prec{n}} 
	\qquad\
	\mathcal{E}_\succ(a) :=\un + \sum_{n > 0}  a^{\succ n}.
$$
They are respectively the formal solutions of the two half-shuffle fixed point equations 
\begin{equation}
\label{recursion}
	X=\un + a\prec X \ \qquad\ Y= \un + Y \succ a.
\end{equation}

\begin{lem}\label{lem:inverse}
Let $D$ be a shuffle algebra, and $\overline D=D \oplus \mathbb{K}.\un$ its unital augmentation. 

1) For $a \in D$, the product of $X:=\mathcal{E}_{\prec}(a)$ and $Y:=\mathcal{E}_{\succ}(-a)$ is $Y * X= X * Y = \un$, so that  $\mathcal{E}^{*-1}_\prec(a)=\mathcal{E}_{\succ}(-a)$. We have therefore 
\begin{equation}
\label{inverse1}
	Y=X^{*-1}=\sum\limits_{n\geq 0}(-1)^n{(\mathcal{E}_\prec (a)-\un)}^{* n}.
\end{equation}

2) For $a \in D$ and $X'=X'(a) :=\mathcal{E}_\prec (a)-\un$, we have
\begin{equation}
\label{leftLog1}
	a=X'\prec \big(\sum\limits_{n\geq 0}(-1)^n{X'}^{* n}\big).
\end{equation}
Analogously, for $Y'=Y'(a) :=\mathcal{E}_\succ (a)-\un$, we have
\begin{equation}
\label{leftLog2}
	a=\big(\sum\limits_{n\geq 0}(-1)^n{Y'}^{* n}\big) \succ Y'.
\end{equation}
\end{lem}

\begin{proof}
For proofs and more details see, for instance, \cite{EFP15,EFP16,EFP18}. 
\end{proof}

\begin{defn}\label{def:left-right-log}
Let $D$ be a shuffle algebra, and $\overline D$ its unital augmentation. For $x \in D$ define the left half-shuffle logarithm
\begin{equation}
\label{left-log}
	  \mathcal{L}_{\prec}(\un + x) := x \prec \big(\sum\limits_{n \geq 0}(-1)^nx^{* n}\big),
\end{equation}
and the right half-shuffle logarithm
\begin{equation}
\label{right-log}
	  \mathcal{L}_{\succ}(\un + x) := \big(\sum\limits_{n \geq 0}(-1)^nx^{* n}\big) \succ x.
\end{equation}
\end{defn}

For the following theorem we define the pre-Lie Magnus expansion \cite{EFM09} in terms of the recursion
\begin{equation}
\label{preLieMagnus}
	\Omega'(a) := \frac{L_{\Omega' \rhd}}{\mathrm{e}^{L_{\Omega' \rhd}}-\id}(a)
            =\sum\limits_{m\ge 0} \frac{b_m}{m!}\ L^{(m)}_{\Omega' \rhd}(a)
            =a - \frac{1}{2} a \rhd a + \sum\limits_{m\ge 2} \frac{b_m}{m!}\ L^{(m)}_{\Omega' \rhd}(a),
\end{equation}
where the $b_l$'s are the Bernoulli numbers. For $a \in D$ we define the map
\begin{equation}
\label{eq:W}
	W'(a) := \frac{\mathrm{e}^{L_{a \rhd}} - \id}{{L_{a \rhd}}}(a)= a + \frac 12 a\rhd a + \frac 16 a\rhd(a\rhd a) + \cdots.
\end{equation}
The bijection $W'$ is the compositional inverse of $\Omega'$, i.e., $W'\circ \Omega' = \id =  \Omega'  \circ W'$.

\begin{thm}\label{thm:pre-LieMagnus}
The left and right half-shuffle exponentials, $\mathcal{E}_\prec (a)$ respectively $\mathcal{E}_\succ (a)$, satisfy
\begin{equation}
\label{shuffleexp-solution}
	\mathcal{E}_\prec (a) = \exp^*\!\big(\Omega'(a)\big)
	\ \qquad
	\mathcal{E}_\succ (a) = \exp^*\!\big(-\Omega'(-a)\big).
\end{equation}
\end{thm}

In the commutative case, i.e., when $L_{a\succ} = R_{\prec a}$ for $a \in D$, the map $\Omega'$ reduces to the identity map. Hence, in a commutative shuffle algebra the two fixed point equations in \eqref{recursion} coincide and the solution is given by $X= \exp^*\!(a).$ 

From \eqref{shuffleexp-solution} we deduce a key identity in shuffle algebra connecting the three exponentials.

\begin{lem}\label{lem:transforming}
Let $D$ be a shuffle algebra, and $\overline D$ its augmentation by the unit $\un$. For $a \in D$ the following identity holds 
\begin{equation}
 \label{transforming}
	\mathcal{E}_{\prec}\big(W'(a)\big)	
	= \exp^*\!(a)  
	= \mathcal{E}_{\succ}\big(-W'(-a)\big).
\end{equation}
\end{lem} 

The interplay between the pre-Lie Magnus expansion and its inverse becomes most intriguing when combining it with the Baker--Campbell--Hausdorff expansion. Indeed, one can show that \cite{EFP18,Manchon11}
\begin{equation}
\label{BCHpreLie}
	\mathrm{BCH}\big(\Omega'(a),\Omega'(b)\big)=\Omega'(a \# b),
\end{equation}
where 
\begin{equation}
\label{BCHproduct1}
	a \# b 
	= W' \big(\mathrm{BCH} \big(\Omega'(a),\Omega'(b)\big)\big).
\end{equation}
From $W'(a) = \mathrm{e}^{L_{a\rhd}}\un - \un$ it follows that
$$
	W'(a) \# W'(b) 
	= W'\big(\mathrm{BCH} (a,b)\big)
	=\mathrm{e}^{L_{a\rhd}}\mathrm{e}^{L_{b\rhd}}\un-\un.
$$ 
Hence $W'(a) \# W'(b) = W'(a) + \mathrm{e}^{L_{a\rhd}}W'(b)$. This yields the formula 
\begin{equation}
\label{BCHproduct2}
	a \# b = a + \mathrm{e}^{L_{\Omega'(a) \rhd}}b.
\end{equation}
We can use $L_{a\rhd}(b)=a \rhd b=(L_{a \succ} - R_{\prec a})(b)$ to rewrite \eqref{BCHproduct2}
\begin{align}
	a \# b &
	= a + \mathrm{e}^{L_{\Omega'(a) \succ}}\mathrm{e}^{R_{ \prec\Omega'(a)}}b\nonumber\\ 
	&= a + \exp^*\!\big(\Omega'(a)\big) \succ b \prec \exp^*\!\big(-\Omega'(a)\big) \nonumber\\ 
	&= a + \mathcal{E}_{\prec}(a) \succ b \prec \mathcal{E}^{*-1}_{\prec}(a).  \label{BCHproduct3}
\end{align}
We used shuffle relation \eqref{A1} which implies for $n \ge 0$, that $(\cdots((b \prec  \Omega'(a) ) \prec  \Omega'(a) ) \prec \cdots )\prec \Omega'(a) = b \prec  ( \Omega'(a)^{*n} )$. Likewise, from \eqref{A3} it follows that $\Omega'(a) \succ ( \cdots \succ (\Omega'(a) \succ(  \Omega'(a) \succ b ))\cdots) = ( \Omega'(a)^{*n} ) \succ b$. 

Using \eqref{BCHpreLie} we note for $a,b \in D$ that
\begin{align*}
	\mathcal{E}_{\prec}(a) \ast \mathcal{E}_{\prec}(b) 
	&= \exp^*\!\big(\Omega'(a)\big) \ast \exp^*\!\big(\Omega'(b)\big)\\
	&= \exp^*\!\big(\mathrm{BCH}\big(\Omega'(a),\Omega'(b)\big)\big)\\
	&= \exp^*\!\big(\Omega'(a \# b)\big)\\
	& = \mathcal{E}_{\prec}(a \# b).
\end{align*}
Similarly, we have that $\mathcal{E}_{\succ}(a) \ast \mathcal{E}_{\succ}(b) = \mathcal{E}_{\succ}(-(-b \# -a))$, where  we used a classical property of the Baker--Campbell--Hausdorff series, $\mathrm{BCH}(a,b)=-\mathrm{BCH}(-b,-a)$.
 
\begin{defn}\cite{EFM09,EFP18}\label{def:groupaction}
Let $D$ be a shuffle algebra, and $\overline D$ its unital augmentation. For $x,y \in D$ we define 
\begin{align}
	  y^x:={Ad}_{x}(y):= \mathcal{E}^{*-1}_{\prec}(x) \succ y \prec  \mathcal{E}_{\prec}(x) \label{shuffleaction1}\\
	  y_x:={Ad}^{x}(y):= \mathcal{E}_{\prec}(x) \succ y \prec  \mathcal{E}^{*-1}_{\prec}(x) \label{shuffleaction22}. 
\end{align}
\end{defn}

First, we note that from \eqref{transforming} in Theorem \ref{lem:transforming} we see that the left half-shuffle exponentials in \eqref{shuffleaction1} and \eqref{shuffleaction22} can be expressed in terms of $\exp^*$ as well as the right half-shuffle exponential
\begin{align*}
	y^x 
	&= \exp^*\!(- \Omega'(x)) \succ y \prec \exp^*\!(\Omega'(x)) \\
	&= \mathcal{E}^{*-1}_{\succ}(-W'(-\Omega'(x)) \succ y \prec  \mathcal{E}_{\succ}(-W'(-\Omega(x)).
\end{align*}

Next, we show that the identity $\mathcal{E}^{*-1}_{\prec}(x) =  \mathcal{E}_{\succ}(-x)$ implies that
\begin{align*}
	y^{-x}={Ad}_{-x}(y) &= \mathcal{E}^{*-1}_{\prec}(-x) \succ y \prec  \mathcal{E}_{\prec}(-x)\\
				     &= \mathcal{E}_{\succ}(x) \succ y \prec  \mathcal{E}^{*-1}_{\succ}(x)\\ 
				     &=:\widetilde{Ad}_{x}(y).	
\end{align*}

\begin{prop}\label{prop:adjointstructure}
Let $D$ be a shuffle algebra, and $\overline D$ its unital augmentation. For $x,y,z \in D$ we have 
\begin{equation}
	{Ad}_{x}{Ad}_{y}(z) = {Ad}_{y  \# x }(z).	\label{shuffleaction2}
\end{equation}
\end{prop}

\begin{proof}
We calculate
\begin{align}
	{Ad}_{x}{Ad}_{y}(z)
	&= \mathcal{E}^{*-1}_{\prec}(x) \succ \big( \mathcal{E}^{*-1}_{\prec}(y) \succ z 
		\prec  \mathcal{E}_{\prec}(y) \big) \prec \mathcal{E}_{\prec}(x) \nonumber \\
	&= \big(\mathcal{E}_{\prec}(y) \ast \mathcal{E}_{\prec}(x) \big)^{*-1} \succ z 
		\prec  \big(\mathcal{E}_{\prec}(y) \ast \mathcal{E}_{\prec}(x) \big) \nonumber \\
	&=  \mathcal{E}^{*-1}_{\prec}(y  \# x) \succ z 
		\prec  \mathcal{E}_{\prec}(y  \# x) \nonumber \\
	&= {Ad}_{y  \# x }(z).	\label{shuffleaction2}
\end{align}
\end{proof}

\begin{cor}\label{cor:BCHinverse}
Let $D$ be a shuffle algebra, and $\overline D$ its unital augmentation. For $a,b \in D$ we have that
\begin{equation}
\label{inverseBCH}
	a \# b^a
	=a + b
	=a^{-b}  \# b. 
\end{equation}
\end{cor} 

\begin{proof} 
First, we observe for the product \eqref{BCHproduct2}, that $a \# b=a + {Ad}^{a}(b)$, and that 
\begin{equation}
\label{inverseAdjoint}
	{Ad}^{a}{Ad}_{a}(b)=(b^a)_a=b=(b_a)^a={Ad}_{a}{Ad}^{a}(b).
\end{equation}
From this we deduce that
$$
	a \# b^a = a + {Ad}^{a}{Ad}_{a}(b)=a+b.
$$ 	
Note that we will show the second equality in \eqref{inverseBCH} explicitly in the following theorem.
\end{proof} 
 
\begin{thm}\cite{EFP17,EFP18}\label{thm:shufflefactorization}
Let $D$ be a shuffle algebra, and $\overline D$ its unital augmentation. For $x,y \in D$ we have the following  factorisations
\begin{equation}
\label{freeConv}
	\mathcal{E}_{\prec}(x + y) = \mathcal{E}_{\prec}(x) * \mathcal{E}_{\prec}(y^x)
\end{equation}
and
\begin{equation}
\label{boolConv}
	\mathcal{E}_{\succ}(x + y) = \mathcal{E}_{\succ}(x^{-y}) * \mathcal{E}_{\succ}(y).
\end{equation}
\end{thm}

\begin{proof}
The proof of both \eqref{freeConv} and \eqref{boolConv} follows from \eqref{BCHpreLie} together with \eqref{BCHproduct3}. We verify \eqref{boolConv} explicitly. 
\allowdisplaybreaks
\begin{align*}
	\mathcal{E}_{\succ}(x^{-y}) * \mathcal{E}_{\succ}(y)
	&=  \exp^*\!\big(-\Omega'(-x^{-y})\big) \ast \exp^*\!\big(-\Omega'(-y)\big)\\
	&=  \exp^*\!\big(\mathrm{BCH}\big(-\Omega'(-x^{-y}),-\Omega'(-y)\big)\big)\\ 
	&=  \exp^*\!\big(-\mathrm{BCH}\big(\Omega'(-y),\Omega'(-x^{-y})\big)\big)\\ 
	&= \exp^*\!\big(-\Omega'(-y \# - x^{-y})\big)\\
	&= \exp^*\!\big(-\Omega'(-y - (x^{-y})_{-y})\big)\\   
	&= \exp^*\!\big(-\Omega'(-y - x)\big)\\
	&= \mathcal{E}_{\succ}(x+y).  
\end{align*}
\end{proof}


\section{Monotone, free and boolean cumulants}
\label{sec:cumulants}

Let us return to Example \ref{ex:HopfAlg} and the Hopf algebra $H=T(T_+(A))$ in Theorem \ref{thm:HA}. Recall that it is connected, graded, non-cocommutative, and non-commutative. Its antipode $S \in \mathrm{End}_\mathbb{K}(H,H)$, i.e., the inverse of the identity $\id \in \mathrm{End}_\mathbb{K}(H,H)$ with respect to the convolution product defined on $\mathrm{End}_\mathbb{K}(H,H)$ in terms of the coproduct \eqref{HopfAlg} on $H$, is given by
\begin{equation}
\label{antipode}
	S=\sum_{i \ge 0} (-1)^i P^{\ast i}.
\end{equation}
The linear map $P:= \id - e$ is the augmentation projector, that is, $P(\un)=0$ and $P=\id$ on the kernel of the counit, $H_+=T_+(T_+(A))$. 

\begin{defn}\label{def:group}
A character $\Phi \in \mathrm{Lin}(H,\mathbb{K})$ is a unital multiplicative map, i.e., $\Phi(\un)=1$ and $\Phi(w | w')=\Phi(w)\Phi(w')$, for $w,w' \in H_+$. An infinitesimal character $\kappa \in \mathrm{Lin}(H,\mathbb{K})$ is a map such that  $\kappa(\un)=0$ and $\kappa(w | w')=0,$ for $w,w' \in H_+$. 
\end{defn}

Recall that the set $G \subset \mathrm{Lin}(H,\mathbb{K})$ of characters forms a group with respect to the convolution product \cite{FGB05,Manchon08}. The convolution inverse of a character $\Phi \in G$ is $\Phi^{*-1}:=\Phi \circ S$. The space $g \subset \mathrm{Lin}(H,\mathbb{K})$ of infinitesimal characters forms a Lie algebra for the Lie bracket $[\alpha,\beta]:=\alpha \ast \beta - \beta \ast \alpha$. The logarithm and exponential maps, $\exp^*$ and $\log^*$, are set isomorphisms between the group $G$ and its Lie algebra $g$. Recall that for any infinitesimal character $\alpha \in g$ and any word $w \in T_+(A)$ of finite length $|w|$, the exponential reduces to a finite sum, i.e., $\exp^*\!(\alpha)(w) = \sum_{j= 1}^{|w|} \frac{1}{j!}\alpha^{* j}(w)$. The same holds for the logarithm, $\log^*(e +\alpha)(w)=\sum_{l= 1}^{|w|}{(-1)^{l-1}\over l}\alpha^{* l}(w)$. For any $\alpha \in g$, the left and right half-shuffle exponentials, $\mathcal{E}_{\succ}(\alpha)$ respectively $\mathcal{E}_{\prec}(\alpha)$, also reduce to finite sums when applied to a word $w \in T_+(A)$ of finite length, i.e., $\mathcal{E}_{\prec}(\alpha)(w)=\sum_{j=1}^{|w|} \alpha^{\prec j}(w)$, and similarly for $\mathcal{E}_{\succ}(w)$. 

Both half-shuffle exponentials provide as well natural bijections between $G$ and $g$ \cite{EFP17,EFP18}. 
It follows that for $\Phi \in G$ there exist unique infinitesimal characters $\alpha$, $\beta$, $\gamma$ in $g$ such that 
\begin{equation}
\label{cumulantrelations}
	\Phi = \exp^*\!(\rho) = \mathcal{E}_\prec(\alpha) = \mathcal{E}_\succ(\beta).
\end{equation}
From this identity together with Theorem \ref{thm:pre-LieMagnus} the following relations between $\alpha$, $\beta$, $\gamma$ in $g$ can be deduced
\begin{equation}
\label{linked}
	\alpha=W'(\rho),\qquad \beta = - W'(-\rho)
\end{equation}
from which $\alpha=W'(- \Omega'(-\beta))$ follows (see \cite{EFP17} for details). 

\smallskip

We now consider $H=T(T_+(A))$ where $(A,\varphi)$ is supposed to be a non-commutative probability space, i.e., a unital $\mathbb{K}$-algebra $A$ with map $\varphi \colon  A \to \mathbb{K}$, and $\varphi(1_A)=1$. See \cite{NicaSpeicher06} for details. First, $\varphi$ is extended to a linear map $\phi$ from $T_+(A)$ to $\mathbb{K}$ by defining $\phi(a_1 \cdots a_n):=\varphi(a_1 \cdot_A \cdots \cdot_A a_n)$. Then $\phi$ is extended to a character $\Phi$ on $H$. For a word $w=a_1 \cdots a_n \in T_+(A)$, the $n$-th order multivariate moment is defined by 
$$
	m_n(a_1,\ldots,a_n):=\varphi(a_1 \cdot_A \cdots \cdot_A a_n)=\Phi(w).
$$ 
From \cite{EFP15,EFP16,EFP18} the next theorem follows. 

\begin{thm} \cite{EFP18} 
Let $(A,\varphi)$ be a non-commutative probability space with unital map $\varphi \colon A \to \mathbb{K}$ and $\Phi$ its extension to $H$ as a character. Let $\rho$, $\kappa$, $\beta$ in $g$ be infinitesimal characters defined in terms of the shuffle algebra identity
\begin{equation}
\label{linkCumulants}
	\Phi = \exp^*\!(\rho) 
			= \mathcal{E}_\prec(\kappa) 
			= \mathcal{E}_\succ(\beta).
\end{equation} 
For the word $w=a_1 \cdots a_n \in T_+(A)$ we set $h_n(a_1,\ldots,a_n)=\rho(w)$,  $k_n(a_1,\ldots,a_n)=\kappa(w)$, and $r_n(a_1,\ldots,a_n)=\beta(w)$. The maps $h_n$, $k_n$, $r_n$ identify respectively with multivariate free, boolean and monotone cumulants and we obtain the following multivariate moment-cumulant relations
\begin{enumerate}
\item[i)] Free moment-cumulant relation \cite{NicaSpeicher06}:
\begin{equation}
\label{free}
	\mathcal{E}_\prec(\kappa)(w) = \sum_{j=1}^n \kappa^{\prec j}(w) = \sum_{\pi \in NC_n} k_\pi(a_1, \ldots, a_n ),
\end{equation}
where $k_\pi(a_1, \ldots, a_n):=\prod_{\pi_i \in \pi} \kappa(a_{\pi_i}).$

\item[ii)] Boolean moment-cumulant relation \cite{Speicher97b}:
\begin{equation}
\label{boolean}
	\mathcal{E}_\succ(\beta)(w) 
					= \sum_{j=1}^n \beta^{\succ j}(w) 
					=\sum_{I \in B_n} r_I(a_1, \ldots, a_n),
\end{equation}
where  $r_I(a_1, \ldots, a_n):=\prod_{l_k \in I} \beta(a_{l_k}).$
 
\item[iii)] Monotone moment-cumulant relation \cite{HasebeSaigo11}:
\begin{equation}
\label{monotone2}
	\exp^*\!(\rho)(w) 
		= \sum_{j=1}^n \frac{\rho^{* j}(w)}{j!} 
		= \sum\limits_{\gamma \in NC_n} \frac{1}{\tau(\gamma)!} h_{\gamma}(a_1, \ldots, a_n),
\end{equation}
The tree factorial $\tau(\gamma)!$ corresponds to the forest $\tau(\gamma)$ of rooted trees encoding the nesting structure of the non-crossing partition $\gamma \in NC_n$ \cite{Arizmendi15}, and  $h_\pi(a_1, \ldots, a_n):=\prod_{\pi_i \in \pi} \rho(a_{\pi_i})$. 

\end{enumerate}
We call the Lie algebra elements $\rho$, $\kappa$, $\beta \in g$ the monotone, free and boolean infinitesimal cumulant characters, respectively. 
\end{thm}

Note that in all three cases the last equality follows from evaluating the lefthand side on a word of finite length. The next result will be useful.
 
\begin{prop}\label{cor:boolean}
Let $\nu, \tau \in g$ be the free and boolean infinitesimal characters of the state $\Psi=\mathcal{E}_\prec(\nu)=\mathcal{E}_\succ(\tau) \in G$. Following \eqref{shuffleaction2} we deduce from $\Psi^{*-1}=\mathcal{E}^{*-1}_\succ(\tau)=\mathcal{E}_\prec(-\tau)$ that
$$
	\mu_\nu	= {Ad}^{\nu}(\mu)
			=\Psi \succ  \mu \prec \Psi^{*-1} 
			= \mathcal{E}^{*-1}_\prec(-\tau) \succ  \mu \prec \mathcal{E}_\prec(-\tau)
			= {Ad}_{-\tau}(\mu)
			= \mu^{-\tau}.
$$
\end{prop}

 From \eqref{linkCumulants} and \eqref{transforming} it follows that monotone, free and boolean cumulants are related. This implies that one can express monotone, free, and boolean cumulants in terms of each other. See \cite{Arizmendi15} for details. We consider the following lemma, which will be useful in describing these relations. From \cite{Arizmendi15} we recall that an irreducible non-crossing partition is a non-crossing partition of the set $[n]$ with $1$ and $n$ being in the same block. The set of  irreducible non-crossing partitions is denoted by $NC^{irr}_n$. 

\begin{lem}\label{lem:adjoint}
Let $\mu, \nu, \tau$ be infinitesimal characters  in $g$ and $\Psi=\mathcal{E}_\prec(\nu)=\mathcal{E}_\succ(\tau) \in G$ (so that $\nu$ is the free cumulant and $\tau$ the boolean infinitesimal cumulant character associated to $\Psi$). 
The following formula holds for the infinitesimal character $\mu^\nu={Ad}_{\nu}(\mu)=\Psi^{*-1} \succ \mu \prec \Psi$ evaluated on a word $w=a_1 \cdots a_n \in T_+(A)$ of length $n$
\begin{align}
\label{one}
	 \mu^\nu(w)&=\sum_{1,n \in S \subseteq [n]} \mu(a_S)\Psi(a_{J_{[n]}^S})\\
	 		 &=\sum_{\pi \in NC^{irr}_n} \prod_{\mathrm{outer} \atop  \pi_1 \in \pi} \mu_{|\pi_1|}(a_{\pi_1})\prod_{\mathrm{inner} \atop  \pi_i \in \pi} \nu_{|\pi_i|}(a_{\pi_i}) \label{one1}. 
\end{align}

On the other hand, from Proposition \ref{cor:boolean} it follows that $\mu_\nu=\mu^{-\tau}$ evaluated on a word of length $n$, $w=a_1 \cdots a_n \in T_+(A)$, gives
\begin{align}
\label{two}
	\mu_\nu(w)=\mu^{-\tau}(w)&=\sum_{1,n \in S \subseteq [n]} \mu(a_S)\Psi^{*-1} (a_{J_{[n]}^S})\\
	 		 &=\sum_{\pi \in NC^{irr}_n} (-1)^{|\pi|-1} \prod_{\mathrm{outer} \atop  \pi_1 \in \pi} \mu_{|\pi_1|}(a_{\pi_1})\prod_{\mathrm{inner} \atop  \pi_i \in \pi} \tau_{|\pi_i|}(a_{\pi_i}) \label{two1}. 
\end{align}
\end{lem} 

Notice that we used for notational convenience the symbol $\prod_{\mathrm{outer} \atop  \pi_1 \in \pi}$ although the product is trivial and involves only one block.

\begin{proof} 
We follow the proof given in \cite{EFP18} by using induction on the length of words. Let $w=a_1 \cdots a_n \in T(A)$ and let $\mu \in g$ be an infinitesimal character. The expression $\mu^\nu={Ad}_{\nu}(\mu)=\Psi^{*-1} \succ \mu \prec \Psi$ is equivalent to $\Psi \succ \mu^\nu = \mu \prec \Psi$, such that  
\begin{align}
	\Psi \succ \mu^\nu (w) &= \mu^\nu(w) + \sum_{j=1}^{n-1} \Psi(a_{j+1} \cdots a_n)\mu^\nu(a_1 \cdots a_j)
	=\sum_{1 \in S \subseteq [n]} \mu(a_S) \Psi(a_{J^S_{[n]}}).
\end{align}
This implies 
\allowdisplaybreaks
\begin{align}
	\mu^\nu(w) &=\sum_{1 \in S \subseteq [n]} \mu(a_S) \Psi(a_{J^S_{[n]}})
		- \sum_{j=1}^{n-1} \Psi(a_{j+1} \cdots a_n)\mu^\nu(a_1 \cdots a_j)\\
		&= \sum_{1,n \in S \subseteq [n]} \mu(a_S) \Psi(a_{J^S_{[n]}})
		 + \sum_{1 \in S \subset [n] \atop n \notin S} \mu(a_S) \Psi(a_{J^S_{[n]}})
		 - \sum_{j=1}^{n-1} \Psi(a_{j+1} \cdots a_n)\mu^\nu(a_1 \cdots a_j) \label{calc1}.
\end{align}
A simple calculation for a single letter $a \in A$ shows that $\mu^\nu(a)=\mu(a)$. For a word of length $n=2$ we find 
\begin{align*}
	\mu^\nu(a_1a_2) 
		&=(\Psi^{*-1} \succ  \mu \prec \Psi)(a_1a_2) \\
		&=\Psi^{*-1}(a_2)(\mu \prec \Psi)(a_1) + (\mu \prec \Psi)(a_1a_2)\\
		&= -\Psi(a_2)\mu(a_1) + \mu(a_1a_2) + \mu(a_1)\Psi(a_2)\\
		&=\mu(a_1a_2). 
\end{align*}
Using induction we write $\mu^\nu(a_1 \cdots a_j)= \sum_{1,j \in S \subseteq [j]} \mu(a_S) \Psi(a_{J^S_{[j]}})$ in \eqref{calc1}. Then
\begin{align}
	\mu^\nu(w) &= \sum_{1,n \in S \subseteq [n]} \mu(a_S) \Psi(a_{J^S_{[n]}})
		 + \sum_{1 \in S \subset [n] \atop n \notin S} \mu(a_S) \Psi(a_{J^S_{[n]}})\nonumber\\
		 &- \sum_{j=1}^{n-1} \Psi(a_{j+1} \cdots a_n) \sum_{1,j \in T \subset [j]} \mu(a_T) \Psi(a_{J^T_{[j]}}) \nonumber\\
		 &=\sum_{1,n \in S \subseteq [n]} \mu(a_S) \Psi(a_{J^S_{[n]}}). \label{irreducible}
\end{align}
Here we used that 
$$
	\sum_{1 \in S \subset [n] \atop n \notin S} \mu(a_S) \Psi(a_{J^S_{[n]}})=
		  \sum_{j=1}^{n-1} \Big(\sum_{1,j \in T \subseteq [j]} \mu(a_T) \Psi(a_{J^T_{[j]}})\Big)\Psi(a_{j+1} \cdots a_n).
$$

The formulation in terms of irreducible non-crossing partitions in \eqref{one1} follows from the fact that the sum on the righthand side of \eqref{irreducible} ranges over subsets $S \subseteq [n]$ which always contain both the elements $1$ and $n$. 

For $\mu_\nu={Ad}^{\nu}(\mu)=\Psi \succ  \mu \prec \Psi^{*-1}$ we recall that, as $\tau$ is the boolean infinitesimal cumulant character of $\Psi \in G$, we have $\Psi^{*-1} =\mathcal{E}^{*-1}_\succ(\tau) =\mathcal{E}_\prec(- \tau)$. Therefore, $\mu_\nu=\mu^{-\tau}={Ad}_{-\tau}(\mu)=\mathcal{E}^{*-1} _\prec(- \tau) \succ  \mu \prec \mathcal{E}_\prec(- \tau)$. Following the same argument as before 
this yields
$$
	\mu^{-\tau}(w)  =\sum_{1,n \in S \subseteq [n]} \mu(a_S) \Psi^{*-1}(a_{J^S_{[n]}})
			    =\sum_{1,n \in S \subseteq [n]} \mu(a_S) \mathcal{E}_\prec(- \tau)(a_{J^S_{[n]}}).
$$
Therefore, we have for any word $w \in T(A)$ that $\Psi^{*-1}(w) =  \mathcal{E}_\prec(- \tau)(w) = \sum_{\pi \in NC_n} (-1)^{|\pi|}\tau_\pi(w)$, which implies the coefficient $(-1)^{|\pi| -1}$ on the righthand side in \eqref{two1}. 
\end{proof}

For instance, from this lemma and \eqref{linkCumulants} we deduce immediately the relation between boolean and free cumulants \cite{EFP18}. Indeed, $\mathcal{E}_\prec(\kappa) = \mathcal{E}_\succ(\beta)$ implies that $\Phi \succ \beta = \kappa \prec \Phi$, which yields $\beta = \Phi^{*-1} \succ \kappa \prec \Phi$ and $\kappa = \Phi \succ \beta \prec \Phi^{*-1}$. Therefore
\begin{align*}
	 \beta(w)&=\sum_{1,n \in S \subseteq [n]} \kappa(a_S)\Phi(a_{J_{[n]}^S}) 
	 =\sum_{\pi \in NC^{irr}_n} k_{\pi}(a_1,\ldots,a_n), 
\end{align*}
and
\begin{align*}
	\kappa(w)&=\sum_{1,n \in S \subseteq [n]} \beta(a_S)\Phi^{*-1} (a_{J_{[n]}^S})
	=\sum_{\pi \in NC^{irr}_n} (-1)^{|\pi|-1} r_{\pi}(a_1,\ldots,a_n). 
\end{align*}
In the last equation we used that from $\Phi=\mathcal{E}_\succ(\beta)$ it follows that $\Phi^{*-1}=\mathcal{E}^{*-1}_\succ(\beta)=\mathcal{E}_\prec(-\beta)$.


\section{Conditionally free cumulants revisited}
\label{sec:condfree}

Let $\varphi$, $\psi$ be two states on the non-commutative probability space $A$. We denote by $\Phi, \Psi$ their extensions to elements of $G$, and by $\beta,\beta' \in g$ and $\kappa,\kappa' \in g$ the corresponding boolean respectively free infinitesimal cumulant characters, i.e., $\Phi=\mathcal{E}_\succ(\beta)=\mathcal{E}_\prec(\kappa)$ and $\Psi=\mathcal{E}_\succ(\beta')=\mathcal{E}_\prec(\kappa')$. Recall the relation between boolean and free infinitesimal cumulant characters 
\begin{equation}
\label{freebool}
	\beta = \Phi^{*-1} \succ \kappa \prec \Phi \qquad \ \beta' = \Psi^{*-1} \succ \kappa' \prec \Psi.
\end{equation}
In the following we would like to determine the infinitesimal cumulant character $R \in g$ such that
\begin{equation}
\label{freeCond1}
	\beta = R^{\kappa'} = \Psi^{*-1} \succ R \prec \Psi,
\end{equation}
and show that it is related to the c-free cumulants $R_n^{(\varphi,\psi)}$ as the infinitesimal characters $\beta$, $\beta'$ and $\kappa$, $\kappa'$ are related to the corresponding boolean and free cumulants.

Lemma \ref{lem:adjoint} implies immediately that 
\allowdisplaybreaks
\begin{align*}
	 \beta(w)
	 	&=\sum_{1,n \in S \subseteq [n]} R(a_S)\Psi(a_{J_{[n]}^S}) \\
	 	&=\sum_{\pi \in NC^{irr}_n} \prod_{\mathrm{outer} \atop  \pi_1 \in \pi}
		R(a_{\pi_1})
		\prod_{\mathrm{inner} \atop  \pi_i \in \pi} 
		{\kappa'}(a_{\pi_i}).
\end{align*}
Inverting \eqref{freeCond1} gives 
\begin{equation}
\label{freeCond2}  
	R = \Psi \succ \beta \prec \Psi^{*-1},
\end{equation}
which evaluates to
\begin{align*}
	R(w)
		&=\sum_{1,n \in S \subseteq [n]} \beta(a_S)\Psi^{*-1} (a_{J_{[n]}^S})\\
	 	&=\sum_{\pi \in NC^{irr}_n} (-1)^{|\pi|-1} \prod_{\mathrm{outer} \atop  \pi_1 \in \pi} 
		{\beta}(a_{\pi_1})
		\prod_{\mathrm{inner} \atop  \pi_i \in \pi} 
		{\beta'}(a_{\pi_i}). 
\end{align*}
As $\beta=\mathcal{L}_\succ(\Phi)=\Phi^{*-1} \succ (\Phi-e)$, we can express $R$ in terms of the characters $\Phi$ and $\Psi$
\begin{equation}
\label{freeCond3}  
	R =  \Psi \succ \big(\Phi^{*-1} \succ (\Phi-e)\big) \prec \Psi^{*-1}.
\end{equation}
For instance, with the following notation for moments 
$$
	m^{\varphi}_n(a_1,\ldots,a_n)=\Phi(w),\ m^{\psi}_n(a_1,\ldots,a_n)=\Psi(w),
$$ 
and using Lemma \ref{lem:adjoint} together with some shuffle algebra we find quickly that 
\begin{align*}
	\lefteqn{R (a_1a_2a_3)
	= (\Phi^{*-1} \succ (\Phi-e))(a_1a_2a_3) 
		+ (\Phi^{*-1} \succ (\Phi-e))(a_1a_3)\Phi^{*-1}(a_2)}\\ 
		     &= m^{\varphi}_3(a_1a_2a_3) 
		     - m^{\varphi}_1(a_3)m^{\varphi}_2(a_1a_2) 
		     - m^{\varphi}_2(a_2a_3)m^{\varphi}_1(a_1)
		     +m^{\varphi}_1(a_1)m^{\varphi}_1(a_2)m^{\varphi}_1(a_3) \\
		     &\quad - m^{\varphi}_2(a_1a_3)m^{\psi}_1(a_2)
		     +m^{\varphi}_1(a_1)m^{\varphi}_1(a_3)m^{\psi}_1(a_2).
\end{align*}
A closer look at \eqref{freeCond2} respectively \eqref{freeCond3} reveals that from the relation between boolean and free cumulants, expressed on the level of infinitesimal characters by \eqref{freebool}, it follows that
\begin{equation}
\label{freeCond4}  
	R =  (\Phi * \Psi^{*-1})^{*-1} \succ \big((\Phi-e) \prec \Phi^{*-1}\big) \prec (\Phi*\Psi^{*-1}),
\end{equation}
where $\kappa=\mathcal{L}_\prec(\Phi)=(\Phi-e) \prec \Phi^{*-1} \in g$. This yields
\begin{equation}
\label{freeCond5}  
	R =  (\Phi * \Psi^{*-1})^{*-1} \succ \kappa \prec (\Phi*\Psi ^{*-1}).
\end{equation}

In terms of half-shuffle exponentials this gives
\begin{align}
	\Phi &= \mathcal{E}_\succ\big(R^{\kappa'}\big) \\ 
		   &=\mathcal{E}_\succ\big(\Psi^{*-1} \succ R \prec \Psi \big) \label{freeCond6}\\
	           &= \mathcal{E}_\prec\big((\Phi * \Psi^{*-1}) \succ R  \prec (\Phi*\Psi ^{*-1})^{*-1}\big). \label{freeCond7}
\end{align}	 
Observe the change from left half-shuffle exponential to right half-shuffle exponential between equations \eqref{freeCond6} and \eqref{freeCond7}. These half-shuffle exponentials solve the corresponding fixed point equations
\begin{equation}
\label{freeCond8}
	\Phi  = e + \Phi  \succ \big(\Psi^{*-1} \succ R \prec \Psi \big)
\end{equation}
respectively 
\begin{equation}
\label{freeCond9}
	\Phi  = e + \big((\Phi * \Psi^{*-1}) \succ R \prec (\Phi*\Psi ^{*-1})^{*-1}\big)\prec \Phi .
\end{equation}
Again, \eqref{freeCond8} reflects the boolean character of the picture, whereas \eqref{freeCond9} is in the free setting. 

\begin{rmk}{\rm{Observe that for $\Psi=\un$, the half-shuffle fixed point equation \eqref{freeCond8} reduces to $\Phi = e + \Phi \succ R$, which implies that $R=\Phi^{*-1} \succ \kappa \prec \Phi=\beta$. For $\Phi=\Psi$ we deduce that $\Phi = e + R \prec \Phi$, such that $R=\kappa$ is the free infinitesimal cumulant character.
}}\end{rmk}

\begin{prop}\label{prop:Speicher}
For a word $w=a_1 \cdots a_n \in T_+(A)$ of length $n$
\begin{align}
\label{freeCond71}
	\Phi(w)=m^{\varphi}_n(a_1,\ldots,a_n)
	&=\sum_{j=1}^n\sum_{1,j\in S \subseteq [j]} R(a_S) \Psi (a_{J_{[j]}^S}) \Phi(a_{j+1}\dots a_n)\\
	&=\sum_{\pi \in NC_n} \prod_{\mathrm{outer} \atop  \pi_j \in \pi} 
	R(a_{\pi_j})
	\prod_{\mathrm{inner} \atop  \pi_i \in \pi} 
	{\kappa'}(a_{\pi_i}).  \label{freeCond72}
\end{align}
\end{prop}

\begin{proof}
From \eqref{freeCond8} and Lemma \ref{lem:adjoint} it follows for a word $w=a_1 \cdots a_n \in T_+(A)$ that
\begin{align}
\label{freeConda}
	\lefteqn{\Phi(w)
	= \big(\Phi \succ \big(\Psi^{*-1} \succ R \prec \Psi \big)\big)(a_1 \cdots a_n)}\\
	&= \sum_{k=2}^{n} \Phi(a_{k}  \cdots a_{n})\big(\Psi^{*-1} \succ R \prec \Psi \big)(a_1 \cdots a_{k-1})\\
	&= \sum_{k=2}^{n} \Big(
	\sum_{1,k-1 \in S \subseteq [k-1]}R(a_S) \Psi (a_{J_{[k-1]}^S}) \Big) \Phi(a_{k}  \cdots a_{n})\\
	&= \sum_{k=2}^{n} 
	\bigg(\sum_{\pi \in NC^{irr}_{k-1}} \prod_{\mathrm{outer}\atop  \pi_1 \in \pi} 
	R(a_{\pi_1})
	\prod_{\mathrm{inner}\atop  \pi_i \in \pi} {k'}_{|\pi_i|}(a_{\pi_i}) \bigg)\Phi(a_{k}  \cdots a_{n}),
\end{align}
which gives by iteration the expression in  \eqref{freeCond72}. \end{proof}

\begin{cor}
The infinitesimal cumulant character $R \in g$ defined in \eqref{freeCond1} computes the multivariate c-free cumulant:
$$
	R(a_1 \cdots a_n)=R^{(\varphi,\psi)}(a_1,\dots,a_n),
$$
using the notation of Section \ref{sect:condit}.
\end{cor}

\begin{proof}
Equation \eqref{freeCond71} and \eqref{freeCond72} identify indeed with the multivariate generalization of the equations defining c-free cumulants (see Section \ref{sect:condit} and reference \cite{Bozejko96}). \end{proof}

Equation \eqref{freeCond3} then says that c-free cumulants are given by an shuffle adjoint action on the boolean logarithm ($R$-transformation) 
$$
	R=  \Psi \succ \mathcal{L}_\succ(\Phi) \prec \Psi^{*-1}.
$$


\section{Conditionally free convolution}
\label{sect:cfreeconv}

Recall from Section \ref{sect:condit} that the c-free convolution $(\varphi,\psi)$ of two c-free states $(\varphi_1,\psi_1)$ and $(\varphi_2,\psi_2)$, is given on the associated free and c-free cumulants by
$$
	k_n^\psi=k_n^{\psi_1}+k_n^{\psi_2},
$$
$$
	R_n^{(\varphi,\psi)}=R_n^{(\varphi_1,\psi_1)} + R_n^{(\varphi_2,\psi_2)}.
$$
Let us write $\Phi,\Psi$ (respectively $\Phi_1,\Psi_1$, $\Phi_2,\Psi_2$) for the associated characters, respectively $\kappa,\beta,\kappa',\beta'$ (and so on) for the associated free and boolean infinitesimal cumulant characters.

In shuffle group theoretical terms, the c-free convolution of  $\Phi_1 =\mathcal{E}_\succ\big(R_1^{\kappa_1'}\big) = \mathcal{E}_\succ\big(\Psi^{*-1}_1 \succ R_1 \prec \Psi_1\big)$ and  $\Phi_2 =\mathcal{E}_\succ\big(R_2^{\kappa_2'}\big) = \mathcal{E}_\succ\big(\Psi^{*-1}_2 \succ R_2 \prec \Psi_2\big)$
is defined by the resulting state
$$
	\Phi = \mathcal{E}_\succ\big(R^{\kappa'}\big) 
			  =  \mathcal{E}_\succ\big(\Psi^{*-1} \succ R \prec \Psi\big),
$$ 
where
\begin{equation}
\label{Cond-1}
	\kappa'= \kappa_1' + \kappa'_2
\end{equation}
and
\begin{equation}
\label{Cond-2}
	R=R_1 + R_2.
\end{equation}
Condition \eqref{Cond-1} implies that $\Psi=\mathcal{E}_\prec( \kappa_1' + \kappa'_2)$, that is
$$
	\Psi = \mathcal{E}_\prec( \kappa_1') *  \mathcal{E}_\prec({\kappa'_2}^{\kappa_1'}),
$$
which is free additive convolution of $\Psi_1$ with $\Psi_2$, i.e., $\Psi =\Psi_1 \boxplus \Psi_2$. The second condition \eqref{Cond-2} is more involved and shows that c-free convolution is different from free additive convolution in general.  

We shall consider three particular cases to illustrate how shuffle group calculus can be developped. First, we assume that  $\kappa_1' = \kappa'_2=0$, i.e., $\Psi_1$ and $\Psi_2$ are just the shuffle unit, and $\kappa'=0$. In light of \eqref{Cond-2} c-free convolution turns out to be just boolean additive convolution
$$
	\Phi = \mathcal{E}_\succ\big(R_1 + R_2\big)
		=\mathcal{E}_\succ\big(R_1^{-R_2}\big) * \mathcal{E}_\succ\big(R_2\big),
$$
that is, $\Phi = \Phi_1 \uplus \Phi_2$.

Next we consider the case when $\kappa_1 =  \kappa_1'$ and $\kappa_2 =  \kappa'_2$. This means that $\Phi_1=\Psi_1$ and $\Phi_2=\Psi_2$. Then $\Phi_1=\mathcal{E}_\succ\big(\Phi^{*-1}_1 \succ R_1 \prec \Phi_1\big)=\mathcal{E}_\prec(R_1)$ and $\Phi_2=\mathcal{E}_\succ\big(\Phi_2^{*-1} \succ R_2 \prec \Phi_2\big)
=\mathcal{E}_\prec(R_2)$. Now, conditions \eqref{Cond-1} and \eqref{Cond-2} imply that the c-free convolution coincides with the free additive convolution such that
$$
	\Phi = \mathcal{E}_\prec(R_1 + R_1)
	=\mathcal{E}_\succ\big(\Phi^{*-1} \succ (R_1+R_2) \prec \Phi\big),
$$ 
and $\Phi=\Psi$.

Last we consider the case, where $\kappa_1'=0$ and $\kappa_2=\kappa'_2$ corresponding to $\Psi_1=e$ and $\Phi_2=\Psi_2$. We will show that in this case c-free convolution coincides with the convolution, i.e., shuffle product in $H^*$. According to \cite{EFP17}, this amounts to saying that in this case c-free convolution coincides with monotone convolution. First notice that $\Psi=\Psi_2=\Phi_2=\mathcal{E}_\prec (\kappa_2)=\mathcal{E}_\prec (\kappa_2')$, whereas $\Psi_1=e$ implies $\beta_1=R_1$ (so that $\Phi_1=\mathcal{E}_\succ (R_1)$). Similarly, $\Phi_2=\Psi_2$ implies $R_2=\kappa_2$ and $\Phi_2=\mathcal{E}_\prec (R_2)$.
We calculate
\begin{align*}
	\Phi_1\ast\Phi_2&=\mathcal{E}_\succ\big(R_1\big) * \mathcal{E}_\prec\big(R_2\big) 
	= \mathcal{E}_\succ\big(R_1\big) * \mathcal{E}_\succ\big(\Psi^{*-1}_2 \succ R_2 \prec \Psi_2\big)\\
	&=\exp^*\!\big(-\Omega'(-R_1)\big) * \exp^*\!\big(-\Omega'(\Psi^{*-1}_2 \succ (-R_2) \prec \Psi_2)\big)\\
	&=\exp^*\!\big(\mathrm{BCH}\big(-\Omega'(-R_1), - \Omega'(\Psi^{*-1}_2 \succ (-R_2) \prec \Psi_2)\big)\big)\\
	&=\exp^*\!\big(-\mathrm{BCH}\big(\Omega'(\Psi^{*-1}_2 \succ (-R_2) \prec \Psi_2),\Omega'(-R_1)\big)\big)\\
	&=\exp^*\!\big(-\Omega'\big((\Psi^{*-1}_2 \succ (-R_2) \prec \Psi_2)  \# (-R_1)\big)\big).
\end{align*} 
Using $a \# b = a + \mathrm{e}^{L_{\Omega'(a) \rhd}}b$ we calculate the product
$$
	(\Psi^{*-1}_2 \succ (-R_2) \prec \Psi_2)  \# (-R_1)
	= - \Psi^{*-1}_2 \succ R_2 \prec \Psi_2 - \mathrm{e}^{L_{\Omega'(\Psi^{*-1}_2 \succ (-R_2) \prec \Psi_2 ) \rhd}}R_1,
$$
which requires to calculate $ \mathrm{e}^{L_{\Omega'(\Psi^{*-1}_2 \succ (-R_2) \prec \Psi_2 ) \rhd}}R_1$
\begin{align*}
	\lefteqn{ \mathrm{e}^{L_{\Omega'(\Psi^{*-1}_2 \succ (-R_2) \prec \Psi_2 ) \rhd}}R_1
	 = \exp^*\!\big(\Omega'(\Psi^{*-1}_2 \succ (-R_2) \prec \Psi_2)\big) \succ
	 R_1 \prec \exp^*\!\big(-\Omega'(\Psi^{*-1}_2 \succ (-R_2) \prec \Psi_2) \big)}\\
	 &= \big(\exp^*\!\big(-\Omega'(\Psi^{*-1}_2 \succ (-R_2) \prec \Psi_2)\big)\big)^{*-1} \succ
	 R_1 \prec \exp^*\!\big(-\Omega'(\Psi^{*-1}_2 \succ (-R_2) \prec \Psi_2) \big)\\
	 &= \mathcal{E}^{*-1}_\succ\big(\Psi^{*-1}_2 \succ R_2 \prec \Psi_2\big) \succ
	 R_1 \prec \mathcal{E}_\succ\big(\Psi^{*-1}_2 \succ R_2 \prec \Psi_2\big)\\ 
	 &= \Psi_2^{*-1} \succ R_1 \prec \Psi_2,
\end{align*} 
since $\Psi_2=\Phi_2=\mathcal{E}_\prec(R_2)=\mathcal{E}_\succ(\Psi^{*-1}_2 \succ R_2 \prec \Psi_2)$.
Hence, this implies that 
\begin{align*}
	\big(\Psi^{*-1}_2 \succ (-R_2) \prec \Psi_2\big)  \# (-R_1)
	&= - \Psi^{*-1}_2 \succ R_2 \prec \Psi_2 - \Psi_2^{*-1} \succ R_1 \prec \Psi_2\\
	&= - \Psi^{*-1}_2 \succ (R_2 + R_1) \prec \Psi_2.
\end{align*} 
This then yields
\begin{align*}
	\mathcal{E}_\succ\big(R_1\big) * \mathcal{E}_\prec\big(R_2\big) 
	&=\exp^*\!\big(-\Omega'\big((\Psi^{*-1}_2 \succ (-R_2) \prec \Psi_2)  \# (-R_1)\big)\big)\\
	&=\exp^*\!\big(-\Omega'\big(- \Psi^{*-1}_2 \succ (R_2 + R_1) \prec \Psi_2\big)\big)\\
	&= \mathcal{E}_\succ(\Psi^{*-1}_2 \succ (R_2 + R_1) \prec \Psi_2).
\end{align*} 
Hence, c-free convolution identifies in that case with the shuffle convolution product. We get  $\Psi=\Psi_1 \boxplus\Psi_2=\Psi_2$ as $\Psi_1=e$ and 
$$
	\Phi= \mathcal{E}_\succ(\Psi^{*-1} \succ (R_1 + R_2) \prec \Psi)=\Phi_1\ast\Phi_2.
$$

\vspace{1cm}

\end{document}